\documentclass [twoside,reqno,12pt] {amsart}
\usepackage{amscd,amssymb,epsf,amsmath,amsfonts,eepic}
\usepackage[toc,page]{appendix}
\usepackage{a4}

\newtheorem{thm}{Theorem}[section]

\newtheorem{lem}[thm]{Lemma}
\newtheorem{prop}[thm]{Proposition}

\theoremstyle{definition}

\theoremstyle{remark}

\numberwithin{equation}{section}

\renewcommand{\Im}{\hbox{Im}\,}

\newcommand{\C}{\mathbb{C}}

\newcommand{\R}{\mathbb{R}}

\newcommand{\supp}{\operatorname{supp}}

\def\hat{\widehat}
\def\tilde{\widetilde}
\def \bfo {\begin {eqnarray*} }
\def \efo {\end {eqnarray*} }
\def \ba {\begin {eqnarray*} }
\def \ea {\end {eqnarray*} }
\def \beq {\begin {eqnarray}}
\def \eeq {\end {eqnarray}}
\def \supp {\hbox{supp }}

\def \det {\hbox{det}}

\def \p {\partial}

\def\hat{\widehat}
\def\tilde{\widetilde}
\def \bfo {\begin {eqnarray*} }
\def \efo {\end {eqnarray*} }
\def \ba {\begin {eqnarray*} }
\def \ea {\end {eqnarray*} }
\def \beq {\begin {eqnarray}}
\def \eeq {\end {eqnarray}}
\def \supp {\hbox{supp }}

\def \det {\hbox{det}}

\def \p {\partial}

\begin{document}

 \title[Determining electrical and heat transfer parameters]{Determining electrical and heat transfer parameters using coupled boundary measurements}

\author[Krupchyk]{Katsiaryna Krupchyk}

\address
        {K. Krupchyk, Department of Mathematics and Statistics \\
         University of Helsinki\\
         P.O. Box 68 \\
         FI-00014   Helsinki\\
         Finland}

\email{katya.krupchyk@helsinki.fi}

\author[Lassas]{Matti Lassas}

\address
        {M. Lassas, Department of Mathematics and Statistics \\
         University of Helsinki\\
         P.O. Box 68 \\
         FI-00014   Helsinki\\
         Finland}

\email{matti.lassas@helsinki.fi}

\author[Siltanen]{Samuli Siltanen}

\address
        {S. Siltanen, Department of Mathematics and Statistics \\
         University of Helsinki\\
         P.O. Box 68 \\
         FI-00014   Helsinki\\
         Finland}

\email{samuli.siltanen@helsinki.fi}

\maketitle

\begin{abstract} 
Let $\Omega\subset\R^n$, $n\ge 3$, be a smooth bounded domain and consider a coupled system in $\Omega$ consisting of a  
conductivity equation $\nabla \cdot \gamma(x) \nabla u(t,x)=0$ and an 
anisotropic heat equation 
$
\kappa^{-1}(x)\p_t\psi(t,x)=\nabla\cdot (A(x)\nabla \psi(t,x))+(\gamma\nabla u(t,x))\cdot \nabla u(t,x), \quad t\ge 0$.  
It is shown that the coefficients $\gamma$, $\kappa$ and $A=(a_{jk})$ are uniquely determined from the knowledge of the boundary map $u|_{\p \Omega}\mapsto \nu\cdot A\nabla \psi|_{\p \Omega}$, where $\nu$ is the unit outer normal to $\p \Omega$. 

The coupled system models the following physical phenomenon. Given a fixed voltage distribution, maintained on the boundary $\partial\Omega$,  an electric current distribution  appears inside $\Omega$. The current in turn acts as a source of heat inside $\Omega$, and the heat flows out of the body through the boundary.  The boundary measurements above then correspond to the map taking a voltage distribution on the boundary to the resulting heat flow through the boundary.  The presented mathematical results suggest a new hybrid diffuse imaging modality combining electrical prospecting and heat transfer-based probing. 

\end{abstract}



\section{Introduction}

\noindent

Let us model a  physical body by a bounded set $\Omega\subset\R^n$, $n\ge 3$, with smooth boundary $\partial\Omega$, and the following spatially varying quantities: heat capacity $c(x)$, density $\rho(x)$, electric conductivity $\gamma(x)$, and (possibly anisotropic) thermal conductivity $A(x)=(a_{jk}(x))$, each defined for $x\in\overline{\Omega}$.

Consider applying a spatially and temporally variable electrical voltage distribution $f(t,x)$ at the boundary $\partial\Omega$ starting at time $t=0$. Then, if there are no sinks or sources of current inside $\Omega$, the electric potential $u(t,x)$ inside the body satisfies the conductivity equation
\begin{equation}
\label{eq_conductivity_int}
\begin{aligned}
\nabla \cdot \gamma \nabla u(t,x)&=0 \qquad \mbox{ for } x\in\Omega\mbox{ and }t\geq 0,\\
u(t,\,\cdot\,)|_{\partial\Omega}&=f(t,\,\cdot\,). 
\end{aligned}
\end{equation}
Equation \eqref{eq_conductivity_int} is often used as a mathematical model for electrical impedance tomography (EIT), where one measures the current through the boundary caused by a family of static voltage distributions $f(t,x)=\phi(x)$ and recovers $\gamma(x)$ from such voltage-to-current map  $\Lambda_\gamma:\phi\mapsto \nu\cdot\nabla u|_{\partial\Omega}$.  Here $\nu$ is the unit outer normal to $\p \Omega$. 
We refer to \cite{Uhl2009} for an extensive survey of the mathematical developments in EIT.    See also \cite{AstPai_2006, AstPaiLassas_2005,  Nach_1996, Syl_1990, SylUhl1986}
for  results in the  two-dimensional case, 
and \cite{GreLasUhl2003, LasUhl01, LasTayUhl03, LeeUhl89,  Nach_1988, PaiPanUhl,   SylUhl1987}  for  results in higher dimensional cases. 
For counterexamples to uniqueness of time-harmonic inverse problems involving very anisotropic and degenerate material parameters, leading to the phenomenon of invisibility,  see 
 \cite{GreKurLasUhl2009, GreKurLasUhl2007, GreLasUhl2003MRL}.

Our aim here is somewhat different as we wish to couple {\em heat conduction} to the problem.
Let us denote the electrical power density inside $\Omega$ by $F$:
\begin{equation}
\label{eq_F_int}
F(t,x)=(\gamma\nabla u(t,x))\cdot \nabla u(t,x).
\end{equation}
Now $F$ acts as a source of heat inside $\Omega$. Assuming that the body is at a constant (zero) temperature at the time $t=0$ when the voltage is first applied, and the surface of the body is kept at that temperature at all times, the temperature distribution $\psi(t,x)$ inside $\Omega$ satisfies the following heat equation:
\begin{equation}\label{eq_heat_int}
\begin{aligned}
\kappa^{-1}(x)\p_t\psi(t,x)&=\nabla\cdot (A(x)\nabla \psi(t,x))+F(t,x) \quad \mbox{ for } x\in\Omega\mbox{ and }t\geq 0,\\
\psi|_{\overline{\R}_+\times \p \Omega}&=0,\quad 
\psi|_{t=0}=0,
\end{aligned}
\end{equation}
where $\kappa(x)=c(x)^{-1}\rho(x)^{-1}$. 
The model  \eqref{eq_conductivity_int}, \eqref{eq_F_int}, \eqref{eq_heat_int} is based on the physical assumption that the heat transfer is so slow that the quasistatic (DC) model for the electric potential \eqref{eq_conductivity_int} is realistic.

Associated to the coupled system \eqref{eq_conductivity_int}, \eqref{eq_F_int},  \eqref{eq_heat_int}, we 
introduce the voltage-to-heat flow map $\Sigma_{\gamma,\kappa,A}$ defined by
\begin{align}\label{boundarymap}
\Sigma_{\gamma,\kappa,A}: f&\mapsto \nu\cdot A\nabla \psi|_{\overline{\R}_+\times \p \Omega}.
\end{align}
The idea is to measure the heat flow through the boundary caused by the heat from the electric current resulting from the applied voltage distribution. 

Our main result is Theorem \ref{thm_main} below, stating that under certain smoothness assumptions,  the coefficients $\gamma$, $\kappa$, and $A$ are uniquely determined from the knowledge of the voltage-to-heat flow map $\Sigma_{\gamma,\kappa,A}$. 

The method of proof of Theorem \ref{thm_main} also outlines a constructive reconstruction procedure for recovering conductivity $\gamma$ from $\Sigma_{\gamma,\kappa,A}$. Namely, it turns out that applying a temporally static voltage distribution $f(t,x)=\phi(x)$ and studying $\Sigma_{\gamma,\kappa,A}f$ at thermal equilibrium ($t\rightarrow\infty$) yields the knowledge of the Dirichlet-to-Neumann map $\Lambda_\gamma\phi$ related to the EIT problem. Then one can recover $\gamma$ using Nachman's reconstruction result \cite{Nach_1988}.

Notice that various hybrid imaging methods have been proposed and analyzed recently. Examples include thermoacoustic and photoacoustic imaging \cite{ABCTF,  BalUhl_2010, KuKu}, combination of electrical and magnetic probing \cite{MREIT,NTT}, electrical and acoustic imaging \cite{GebSch2008} and magnetic and acoustic imaging \cite{MaHe,Ammari2}. Theorem \ref{thm_main} suggests a new hybrid imaging method, utilizing two diffuse modes of propagation: electrical prospecting and heat transfer-based probing. We emphasize that the proposed method recovers complementary information about three different physical properties. We also note that in many applications where one wants to reconstruct the heat transfer parameters $\kappa(x)$ and $A(x)$, the use of electric boundary sources may be easier than controlling the temperature or the heat flux at the boundary.  Concerning inverse problems for the heat equation, we refer to \cite{BukhKlib_1981, Canuto_Kavian_2001,  ImanYam_1998,  Isak_1999, KKL_book, KKLM_2004, Klib_1992, Yamom2009}.

We remark that in practice one might use a measurement setup shown in Figure \ref{fig:electrodes}. However, analysis of such discrete measurements is outside the scope of this paper, and in the mathematical results below we work with the continuum models \eqref{eq_conductivity_int}, \eqref{eq_F_int},   \eqref{eq_heat_int}, and \eqref{boundarymap}.
\begin{figure}
\begin{picture}(300,130)
\epsfxsize=4cm
\put(-10,5){\epsffile{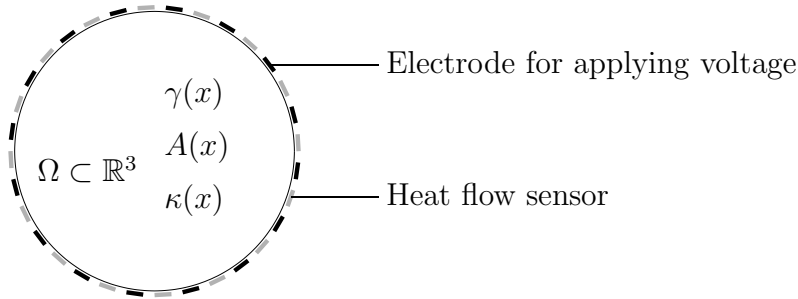}}
\put(2,50){$\Omega\subset \R^3$}
\put(50,80){$\gamma(x)$}
\put(50,60){$A(x)$}
\put(50,40){$\kappa(x)$}
\put(90,94){\line(1,0){41}}
\put(134,91){Electrode for applying voltage}
\put(98,44){\line(1,0){33}}
\put(134,41){Heat flow sensor}
\end{picture}
\caption{\label{fig:electrodes}Schematic illustration of the practical measurement setup motivating the proposed hybrid imaging method. 
The ideal voltage-to-heat flow map $\Sigma_{\gamma,\kappa,A}$ may be approximated in practice by maintaining fixed voltages at the electrodes while measuring heat flow using the interlaced sensors.}
\end{figure}

This paper is organized as follows. In Section \ref{sec:statement} we state our assumptions and results in a mathematically precise form. In Section \ref{sec_auxiliary} we give an auxiliary density result for the conductivity equation. Section \ref{sec:EIT} is devoted to  the reconstruction of the conductivity $\gamma$. The proof of Theorem   \ref{thm_main} is completed in Section \ref{sec_A_kappa}, where we show the identifiability of the heat parameters $\kappa$ and $A$. Finally, Appendix A  is devoted to the recovery of the boundary values of the matrix $A$ from interior--to--boundary measurements, associated to a suitable elliptic boundary value problem. 
This result may be of an independent interest.

\section{Statement of results}\label{sec:statement}

Let $\Omega\subset\R^n$, $n\ge 3$, be a bounded domain with $C^\infty$ boundary. Let $\gamma\in C^\infty(\overline{\Omega})$ be a strictly positive function on $\overline{\Omega}$. 
Then given $f(t,x)\in C^1(\overline{\R}_+,H^{s}(\p \Omega))$, $s\ge 1/2$, on the boundary at time $t\ge 0$, there exists a unique $u\in C^1(\overline{\R}_+,H^{s+1/2}(\Omega))$, which solves the boundary value problem
\begin{equation}
\label{eq_conductivity}
\begin{aligned}
\nabla \cdot \gamma \nabla u(t,x)&=0 \qquad \mbox{ for } x\in\Omega\mbox{ and }t\geq 0,\\
u(t,\,\cdot\,)|_{\partial\Omega}&=f(t,\,\cdot\,), 
\end{aligned}
\end{equation}
see \cite{Grubbbook2009}. 
We have 
\begin{equation}
\label{F_main_text}
F(t,x)=(\gamma\nabla u(t,x))\cdot \nabla u(t,x)
\in  C^1(\overline{\R}_+,L^2(\Omega)),
\end{equation}
provided that $s$ is taken large enough, say $s>(n+1)/2$. In what follows we shall always choose the Sobolev index $s$ in this way. 
Consider the anisotropic heat equation
\begin{equation}\label{eq_heat}
\begin{aligned}
\kappa^{-1}(x)\p_t\psi(t,x)&=\nabla\cdot (A(x)\nabla \psi(t,x))+F(t,x) \quad \mbox{ for } x\in\Omega\mbox{ and }t\geq 0,\\
\psi|_{\overline{\R}_+\times \p \Omega}&=0,\quad 
\psi|_{t=0}=0.
\end{aligned}
\end{equation}
Here $A(x)=(a_{jk}(x))$ is 
a real symmetric $n\times n$ matrix with $a_{j,k}(x)\in C^\infty(\overline{\Omega})$, and  there exists $C_0>0$ such that 
\begin{equation}
\label{eq_A_positive}
\sum_{j,k=1}^n a_{jk}(x)\xi_j\xi_k\ge C_0|\xi|^2,\quad \textrm{for all } (x,\xi)\in \overline{\Omega}\times \R^n.
\end{equation}
 We shall assume that $0<\kappa\in C^\infty(\overline{\Omega})$. 
The operator 
\[
Pv=-\kappa(x)\nabla\cdot (A(x)\nabla v)
\]
is formally self-adjoint in $L^2_\kappa=L^2(\Omega,\kappa^{-1}dx)$ and we have
\[
(Pv,v)_{L^2_\kappa}\ge C_0\| v\|^2_{L^2},\quad v\in C^\infty_0(\Omega),\quad C_0>0.
\]
We also let $P$ denote the Friedrichs extension of the operator $P$ on $C^\infty_0(\Omega)$, so that the domain of the positive self-adjoint operator $P$ is $(H^1_0\cap H^2)(\Omega)$. 

The solution of \eqref{eq_heat} is given by the Duhamel formula
\begin{equation}
\label{eq_duhamel}
\psi(t,x)=\int_0^t (e^{-(t-s)P}\kappa F)(s,x)ds\in C^1(\overline{\R}_+,L^2(\Omega))\cap C(\overline{\R}_+,(H^1_0\cap H^2)(\Omega)),
\end{equation}
see \cite{Grubbbook2009}. 

Associated to the coupled system \eqref{eq_conductivity},  \eqref{F_main_text}, and \eqref{eq_heat}, we 
consider the voltage-to-heat flow map,
\begin{align*}
\Sigma_{\gamma,\kappa,A}: C^1(\overline{\R}_+,H^s(\p \Omega))&\to C(\overline{\R}_+,H^{1/2}(\p \Omega)),\\
 f&\mapsto \nu\cdot A\nabla \psi|_{\overline{\R}_+\times \p \Omega},
\end{align*}
The main result of the paper is as follows.

\begin{thm}
\label{thm_main}

 Assume that $0<\gamma_j\in C^\infty(\overline{\Omega})$, $0<\kappa_j\in C^\infty(\overline{\Omega})$, and $A_j$ are real symmetric $n\times n$ matrices with $C^\infty(\overline{\Omega})$ entries, satisfying \eqref{eq_A_positive}, for $j=1,2$. 
If $\Sigma_{\gamma_1,\kappa_1,A_1}=\Sigma_{\gamma_2,\kappa_2,A_2}$, then $\gamma_1=\gamma_2$, $\kappa_1=\kappa_2$ and $A_1=A_2$. 
\end{thm}

It turns out that in the course of the proof of Theorem \ref{thm_main}, we establish a result for the anisotropic heat equation, which may be of independent interest.  In order to state the result, consider the inhomogeneous initial boundary value problem 
\eqref{eq_heat} for the anisotropic heat equation with an arbitrary source $F\in C^1(\overline{\R}_+,L^2(\Omega))$. 
Define the map,
\begin{align*}
\Xi_{\kappa,A}: C^1(\overline{\R}_+,L^2(\Omega))&\to C(\overline{\R}_+,H^{1/2}(\p \Omega)),\\
 F&\mapsto \nu\cdot A\nabla \psi^F|_{\overline{\R}_+\times \p \Omega},
\end{align*}
where $\psi^F$ is the solution of \eqref{eq_heat}.

\begin{thm}
\label{thm_main_2} Let $\Omega\subset \R^n$, $n\ge 2$, be a bounded domain with $C^\infty$ boundary.  Assume that  $0<\kappa_j\in C^\infty(\overline{\Omega})$, and $A_j$ are real symmetric $n\times n$ matrices with $C^\infty(\overline{\Omega})$ entries, satisfying \eqref{eq_A_positive}, for $j=1,2$.
If $\Xi_{\kappa_1,A_1}=\Xi_{\kappa_2,A_2}$, then  $\kappa_1=\kappa_2$ and $A_1=A_2$. 
\end{thm}

\section{An auxiliary density result}\label{sec_auxiliary}

Let $\Omega\subset\R^n$, $n\ge 3$, be a bounded domain with $C^\infty$ boundary and   let $\gamma\in C^\infty(\overline{\Omega})$ be a strictly positive function on $\overline{\Omega}$.   
We shall need the following density result, which is a quite  straightforward consequence  of \cite{SylUhl1987}. 
\begin{prop}
\label{prop_density_gradients}
The set 
\[
\emph{\textrm{span}}\{\gamma\nabla w_1\cdot\nabla w_2: w_j\in C^\infty(\overline{\Omega}), \nabla\cdot\gamma \nabla w_j=0, j=1,2\} 
\]
is dense in $L^2(\Omega)$. 
\end{prop}

\begin{proof}
Let $f\in L^2(\Omega)$ be  such that
\begin{equation}
\label{eq_density}
\int_{\Omega}f\gamma \nabla w_1\cdot\nabla w_2 dx=0,
\end{equation}
for any solution $w_1, w_2\in C^\infty(\overline{\Omega})$ of the conductivity equation
\begin{equation}
\label{eq_cond}
\nabla\cdot\gamma \nabla w=0.
\end{equation}
We have
\begin{equation}
\label{eq_iden}
\gamma\nabla w_1\cdot \nabla w_2=\frac{1}{2}\nabla \cdot(\gamma \nabla(w_1w_2)). 
\end{equation}
It follows from \cite{SylUhl1987}, see also \cite{BalUhl_2010},
that for any $\rho\in \C^n$ satisfying $\rho\cdot\rho=0$ and $|\rho|\ge 1$ large enough, the conductivity equation \eqref{eq_cond} has a solution
\begin{equation}
 \label{eq_complex_geom_optics}
w_{\rho}(x)=e^{i\rho\cdot x}\gamma^{-1/2}(1+r_\rho),
\end{equation}
 where $r_\rho\in C^\infty(\overline{\Omega})$ satisfies
 \begin{equation}
 \label{eq_complex_geom_optics_2}
  \|r_\rho\|_{H^m(\Omega)}\le \frac{C_m}{|\rho|},\quad m\ge 0.
 \end{equation}
 Here the constant $C_m$ depends on $\Omega$, $n$, and a finite number of derivatives of $\gamma$. 
 
 Given  $\xi\in \R^n$ and $R>0$, according to \cite{SylUhl1987}, there exist $\rho_1,\rho_2\in\C^n$ such that $\rho_j\cdot\rho_j=0$, $ \rho_1+\rho_2=\xi$ and $|\rho_j|\ge R$,  $j=1,2$.

 For the solutions $w_{\rho_1}$ and $w_{\rho_2}$ of the form \eqref{eq_complex_geom_optics}, we get
\begin{align*}
\nabla \cdot(\gamma &\nabla(w_{\rho_1}w_{\rho_2}))=\nabla \cdot(\gamma \nabla(\gamma^{-1}e^{i\xi\cdot x}(1+r_{\rho_1}+r_{\rho_2}+r_{\rho_1}r_{\rho_2})))\\
&=(\nabla\cdot (\gamma\nabla \gamma^{-1})+\gamma(\nabla \gamma^{-1})\cdot i\xi-|\xi|^2)e^{i\xi\cdot x}(1+r_{\rho_1}+r_{\rho_2}+r_{\rho_1}r_{\rho_2})\\
&+ (\gamma\nabla \gamma^{-1}+2i\xi)e^{i\xi\cdot x}\cdot (\nabla r_{\rho_1}+\nabla r_{\rho_2}+\nabla(r_{\rho_1}r_{\rho_2}))\\
&+ e^{i\xi\cdot x} (\Delta r_{\rho_1}+\Delta r_{\rho_2}+\Delta (r_{\rho_1}r_{\rho_2})).
\end{align*}
In view of \eqref{eq_iden}, we may substitute the latter expression into 
\eqref{eq_density} and let $R\to \infty$. Using 
\eqref{eq_complex_geom_optics_2}, we obtain that 
\begin{equation}
\label{eq_density_1}
\int_\Omega (\nabla\cdot (\gamma\nabla \gamma^{-1})+\gamma(\nabla \gamma^{-1})\cdot i\xi-|\xi|^2)e^{i\xi\cdot x} fdx=0\quad \textrm{for all } \xi\in\R^n. 
\end{equation}
Here we shall view 
$\gamma$ as a strictly positive $C^\infty$ function on $\R^n$, which is  equal to a positive constant near infinity. 
The identity \eqref{eq_density_1} is equivalent to
\[
\mathcal{F}_{x\to \xi}(\nabla\cdot (\gamma\nabla \gamma^{-1}) \chi_{\Omega}f -\nabla\cdot(\gamma \nabla(\gamma^{-1})\chi_{\Omega}f) +\Delta( \chi_\Omega f) )=0,
\]
where $\mathcal{F}_{x\to \xi}$ denotes the Fourier transformation and $\chi_\Omega$ is the characteristic function of $\Omega$. 
It follows that $\chi_\Omega f$ is a solution of a second order elliptic equation on $\R^n$ with smooth coefficients. Since it is compactly supported, by unique continuation we conclude that $f\equiv 0$ in $\Omega$. This completes the proof. 
\end{proof}

\section{Recovering the conductivity $\gamma$ from the voltage-to-heat flow map}\label{sec:EIT}

The purpose of this section is to make the first step in the proof of Theorem \ref{thm_main}, by establishing the following result. 
Recall that here $n\ge 3$. 

\begin{prop}
\label{prop_cond}
The voltage-to-heat flow map $\Sigma_{\gamma,\kappa,A}$ determines the conductivity  $\gamma$ uniquely. 
\end{prop}

When proving Proposition \ref{prop_cond}, we 
let $\alpha\in C^\infty(\overline{\R}_+; [0,1])$ be such that 
\[
\alpha|_{t<1/2}=0,\quad \alpha|_{t>1}=1.
\]
Then set $f(t,x)=\alpha(t)h(x)$ with $h\in H^s(\p \Omega)$, $s$ large enough. Using the Duhamel formula \eqref{eq_duhamel}, we shall study the behavior of $\psi(t,x)$ as $t\to +\infty$.
The solution $u$ of \eqref{eq_conductivity} satisfies
\[
u(t,x)=\alpha(t)w_0(x),
\]
where $w_0$ solves 
\begin{equation}
\label{eq_conductivity_w_0}
\begin{aligned}
\nabla \cdot \gamma \nabla w_0(x)&=0 \quad \textrm{in}\quad \Omega,\\
w_0|_{\p\Omega}&=h(x). 
\end{aligned}
\end{equation}
Thus,
\[
F(t,x)=\gamma\alpha^2(t)\nabla w_0\cdot \nabla w_0,
\]
and  \eqref{eq_duhamel} gives
\begin{equation}
\label{eq_psi_1}
\psi(t,x)=\int_0^t \alpha^2(s)e^{-(t-s)P}\gamma\kappa\nabla w_0\cdot \nabla w_0ds.
\end{equation}
\begin{lem}
\label{lem_convergence_psi}
We have
\[
\|\psi(t,\cdot)-P^{-1}(\gamma\kappa\nabla w_0\cdot \nabla w_0)\|_{\mathcal{D}(P)}\to 0, \text{ as }  t\to+\infty.
\]
Here  $\mathcal{D}(P)=(H^2\cap H^1_0)(\Omega)$ is equipped with the graph norm
$\|\psi\|_{\mathcal{D}(P)}=\|\psi\|_{L^2(\Omega)}+\|P\psi\|_{L^2(\Omega)}$.
\end{lem}

\begin{proof} 
We shall first check that 
\begin{equation}
\label{eq_step1}
\|\psi(t,\cdot)-P^{-1}(\gamma\kappa\nabla w_0\cdot \nabla w_0)\|_{L^2(\Omega)}\to 0,\text{ as } t\to+\infty.
\end{equation}
It follows from \eqref{eq_psi_1} that for $t>1$, we have 
\[
\psi(t,x)=\int_0^1 \alpha^2(s)e^{-(t-s)P}\gamma\kappa\nabla w_0\cdot \nabla w_0ds+ 
\int_0^{t-1} e^{-sP}\gamma\kappa\nabla w_0\cdot \nabla w_0ds. 
\]
Using that 
\[
\int_0^{+\infty} e^{-tP}dt=P^{-1},
\]
in the sense of bounded operators on $L^2(\Omega)$, see \cite{Grubbbook2009}, we get with $L^2$-convergence, as $t\to +\infty$,
\[
\int_0^{t-1} e^{-sP}\gamma\kappa\nabla w_0\cdot \nabla w_0ds\to P^{-1}(\gamma\kappa\nabla w_0\cdot \nabla w_0).
\]
On the other hand, 
\begin{align*}
\|\int_0^1 &\alpha^2(s)e^{-(t-s)P}\gamma\kappa\nabla w_0\cdot \nabla w_0ds\|_{L^2(\Omega)}\\
&\le 
\int_0^1 \alpha^2(s)\|e^{-(t-s)P}\gamma\kappa\nabla w_0\cdot \nabla w_0\|_{L^2(\Omega)}ds\\
&\le \int_0^1 \alpha^2(s)\|e^{-(t-s)P}\|_{\mathcal{L}(L^2(\Omega),L^2(\Omega))}\|\gamma\kappa\nabla w_0\cdot \nabla w_0\|_{L^2(\Omega)}ds\to 0, \text{ as } t\to +\infty,
\end{align*}
since an application of the spectral theorem shows that 
\[
\|e^{-(t-s)P}\|_{\mathcal{L}(L^2(\Omega),L^2(\Omega))}\le \sup_{\lambda\in \textrm{spec}(P)}e^{-\lambda(t-s)}\to 0,  \text{ as } t\to+\infty.
\]
 This establishes \eqref{eq_step1}. 

Next we shall show that
\begin{equation}
\label{eq_step2}
\|P\psi(t,\cdot)-\gamma\kappa\nabla w_0\cdot \nabla w_0\|_{L^2(\Omega)}\to 0, \text{ as } t\to+\infty.
\end{equation}
We have for $t>1$, 
\[
P\psi(t,x)=P\int_0^1 \alpha^2(s)e^{-(t-s)P}\gamma\kappa\nabla w_0\cdot \nabla w_0ds+ 
P\int_0^{t-1} e^{-sP}\gamma\kappa\nabla w_0\cdot \nabla w_0ds. 
\]
The formula
\[
P\int_0^t e^{-sP}vds=-e^{-tP}v+v,
\]
see \cite{Grubbbook2009}, implies that
\[
P\int_0^{t-1} e^{-sP}\gamma\kappa\nabla w_0\cdot \nabla w_0ds=-e^{-(t-1)P}\gamma\kappa\nabla w_0\cdot \nabla w_0+ \gamma\kappa\nabla w_0\cdot \nabla w_0,
\]
and 
\[
\|e^{-(t-1)P}\gamma\kappa\nabla w_0\cdot \nabla w_0\|_{L^2(\Omega)}\to 0, \text{ as }  t\to+\infty. 
\]
Finally,
\begin{align*}
\|&P\int_0^1 \alpha^2(s)e^{-(t-s)P}\gamma\kappa\nabla w_0\cdot \nabla w_0ds\|_{L^2(\Omega)}\\
&\le 
\int_0^1 \alpha^2(s)\|Pe^{-(t-s)P}\|_{\mathcal{L}(L^2(\Omega),L^2(\Omega))}\|\gamma\kappa\nabla w_0\cdot \nabla w_0\|_{L^2(\Omega)}ds\to 0, \text{ as } t\to +\infty,
\end{align*}
since
\[
\|Pe^{-(t-s)P}\|_{\mathcal{L}(L^2(\Omega),L^2(\Omega))}\le \sup_{\lambda\in\textrm{spec}(P)}( \lambda e^{-(t-s)\lambda})\to 0, \text{ as } t\to +\infty.
\]
This proves \eqref{eq_step2} and completes the proof of the lemma. 
\end{proof}

Lemma \ref{lem_convergence_psi} implies that as $t\to +\infty$, 
\[
\nu \cdot A\nabla \psi|_{\p\Omega}\to \nu \cdot A \nabla (P^{-1}(\gamma\kappa\nabla w_0\cdot \nabla w_0))|_{\p\Omega}\quad\textrm{in}\ H^{1/2}(\p \Omega).
\]
 Thus, as $t\to +\infty$, we have by a repeated application of the divergence theorem together with the Cauchy-Schwarz inequality,  
 \begin{align*}
 \int_{\p \Omega} \Sigma_{\gamma,\kappa,A}(\alpha(t)h(x))dS=\int_{\p \Omega} \nu\cdot A\nabla \psi dS\to \int_{\p \Omega}\nu\cdot A \nabla (P^{-1}(\gamma\kappa\nabla w_0\cdot \nabla w_0))dS\\
 =\int_{\Omega}\nabla \cdot A \nabla (P^{-1}(\gamma\kappa\nabla w_0\cdot \nabla w_0))dx=-\int_{\Omega} \gamma\nabla w_0\cdot \nabla w_0dx=
 -\int_{\p \Omega} \Lambda_\gamma (h)hdS.
 \end{align*}
Here 
\begin{align*}
\Lambda_\gamma:\ & H^s(\p \Omega)\to H^{s-1}(\p \Omega),\\
&h\mapsto \gamma\p_{\nu} w_0|_{\p \Omega},
\end{align*}
is the Dirichlet-to-Neumann map, associated to the problem \eqref{eq_conductivity_w_0}. Thus, the knowledge of the voltage-to-heat flow map $\Sigma_{\gamma,\kappa,A}$ determines the  Dirichlet-to-Neumann map $\Lambda_\gamma$. It follows from \cite{SylUhl1987} that the isotropic conductivity $\gamma\in C^\infty(\overline\Omega)$ is uniquely determined by $\Sigma_{\gamma,\kappa,A}$.  This completes the proof of Proposition \ref{prop_cond}.

\section{Recovering the heat parameters $A$ and $\kappa$ }\label{sec_A_kappa}

The purpose of this section is to prove  the following result.

\begin{prop}
If $\Sigma_{\gamma,\kappa_1,A_1}=\Sigma_{\gamma,\kappa_2,A_2}$, then $\kappa_1=\kappa_2$ and $A_1=A_2$.
\end{prop}

When determining the conductivity $\gamma$ in the previous section, we were concerned with the power densities $F$ which are supported in the region $t\ge 1/2$ and independent of $t$ near $+\infty$. Here we shall instead concentrate the density $F$ in a small neighborhood of $t=0$.

Let $\chi\in C^\infty(\overline{\R}_+)$ be such that $0\le \chi\le 1$, $\supp(\chi)\subset [0,1]$ and 
\[
\int\chi^2(t)dt=1.
\]
For $\varepsilon>0$, we define
\begin{equation}
\label{eq_chi_epsilon}
\chi_\varepsilon(t)=\varepsilon^{-1/2}\chi(t/\varepsilon),
\end{equation}
so that $\chi_\varepsilon^2(t)\to\delta(t)$, the Dirac measure at $t=0$, as $\varepsilon\to 0$. 
Let $h\in H^s(\p \Omega)$, for $s$ large enough. Then the solution $u=u^{\varepsilon,h}$ of the problem
\begin{align*}
\nabla \cdot \gamma \nabla u(t,x)&=0 \quad \textrm{in}\quad \Omega,\\
u|_{\p\Omega}&=\chi_\varepsilon(t)h(x),
\end{align*}
satisfies 
$u^{\varepsilon,h}(t,x)=\chi_\varepsilon(t)w^h(x)$, where
$w^h(x)$ solves 
\begin{equation}
\label{eq_w_h}
\begin{aligned}
\nabla \cdot \gamma \nabla w^h(x)&=0 \quad \textrm{in}\quad \Omega,\\
w^h|_{\p\Omega}&=h(x).
\end{aligned}
\end{equation}
Let $h,\tilde h\in H^s(\p \Omega)$ and
\[
F^{\varepsilon, h\pm \tilde h}=\gamma \nabla (u^{\varepsilon,h}\pm u^{\varepsilon,\tilde h})\cdot \nabla (u^{\varepsilon,h}\pm u^{\varepsilon,\tilde h})=
\gamma \chi^2_\varepsilon(t)\nabla (w^h\pm w^{\tilde h})\cdot \nabla (w^h\pm w^{\tilde h}).  
\]
Denote by $\psi^{\varepsilon,h\pm \tilde h,j}$ the solution of 
\eqref{eq_heat} with $F=F^{\varepsilon, h\pm \tilde h}$,  $\kappa=\kappa_j$ and $A=A_j$, $j=1,2$.  Set 
\[
\alpha^{(j)}(t,x)=\alpha^{\varepsilon,h,\tilde h,j}(t,x)=\frac{1}{4}(\psi^{\varepsilon,h+\tilde h,j}(t,x)-\psi^{\varepsilon,h-\tilde h,j}(t,x)).
\]
Thus, $\alpha^{(j)}$ is a solution of the following inhomogeneous initial boundary value problem, 
\begin{equation}
\label{eq_heat_alpha}
\begin{aligned}
(\p_t+P_j)\alpha^{(j)}&=\frac{1}{4}\kappa_j(x)\chi_\varepsilon^2(t)\gamma\\
&\bigg(\nabla (w^h+w^{\tilde h})\cdot \nabla (w^h+w^{\tilde h}) - \nabla (w^h-w^{\tilde h})\cdot \nabla (w^h-w^{\tilde h})\bigg)\\
&=\kappa_j(x)\chi_\varepsilon^2(t)\gamma \nabla w^h\cdot \nabla w^{\tilde h},
\quad  \textrm{in}\quad \R_+\times \Omega,\\
\alpha^{(j)}|_{\overline{\R}_+\times \p \Omega}&=0,\quad 
\alpha^{(j)}|_{t=0}=0.
\end{aligned}
\end{equation}
Since 
\[
P_j=-\kappa_j\nabla\cdot A_j\nabla 
\]
 is a self-adjoint positive operator in $L^2_{\kappa_j}(\Omega)$ with the domain $\mathcal{D}(P_j)=(H^1_0\cap H^2)(\Omega)$, the spectrum of 
$P_j$ is discrete, accumulating at $+\infty$,
consisting of eigenvalues of finite multiplicity,
$0<\lambda_1^{(j)}\le \lambda_2^{(j)}\le \cdots\to \infty$.
Associated to the eigenvalues $\lambda_k^{(j)}$ we have the eigenfunctions $\varphi_k^{(j)}\in \mathcal{D}(P_j)$, which
form an orthonormal basis in $L^2_{\kappa_j}(\Omega)$. In what follows we shall assume, as we may, that the eigenfunctions $\varphi_k^{(j)}$ are real-valued. 
Hence, 
\[
\alpha^{(j)}(t,x)=\sum_{k=1}^\infty c^{(j)}_{k,\varepsilon}(t)\varphi_k^{(j)}(x),\quad c^{(j)}_{k,\varepsilon}=(\alpha^{(j)},\varphi_k^{(j)})_{L^2_{\kappa_j}},
\]
with convergence in $C(\overline{\R}_+,\mathcal{D}(P_j))$. Therefore, 
\[
\nu\cdot A_j\nabla \alpha^{(j)}|_{\p \Omega}=\sum_{k=1}^{\infty}c^{(j)}_{k,\varepsilon}(t)(\nu\cdot A_j\nabla \varphi^{(j)}_k)|_{\p \Omega}
\]
with convergence in $H^{1/2}(\p \Omega)$, for each fixed $t\ge 1$.

Next we notice that 
\[
\nu\cdot A_j\nabla \alpha^{(j)}|_{\p \Omega\times\overline{\R}_+}=\frac{1}{4}\Sigma_{\gamma,\kappa_j,A_j}(\chi_\varepsilon (t)(h+\tilde h))-\frac{1}{4}\Sigma_{\gamma,\kappa_j,A_j}(\chi_\varepsilon (t)(h-\tilde h)).
\]
Thus, since $\Sigma_{\gamma,\kappa_1,A_1}=\Sigma_{\gamma,\kappa_2,A_2}$, it follows that for all $t\ge 1$, 
\begin{equation}
\label{eq_bou_100}
\sum_{k=1}^{\infty}c^{(1)}_{k,\varepsilon}(t)(\nu\cdot A_1\nabla \varphi^{(1)}_k)|_{\p \Omega}=\sum_{k=1}^{\infty}c^{(2)}_{k,\varepsilon}(t)(\nu\cdot A_2\nabla \varphi^{(2)}_k)|_{\p \Omega}.
\end{equation}

Here we would like to let $\varepsilon\to 0$.  In order to do so, it will be convenient to obtain an explicit representation of the Fourier coefficients $c^{(j)}_{k,\varepsilon}(t)$.

Set
$d^{(j)}_k=(\gamma \nabla w^h\cdot \nabla w^{\tilde h},\varphi^{(j)}_k)_{L^2}$, where the scalar product is taken in the space $L^2(\Omega,dx)$. 
It follows from 
\eqref{eq_heat_alpha} that
\begin{align*}
\p_t c^{(j)}_{k,\varepsilon}(t)+\lambda_k^{(j)} c^{(j)}_{k,\varepsilon}(t)&=\chi^2_{\varepsilon}(t) d^{(j)}_k,\\
c^{(j)}_{k,\varepsilon}(0)&=0.
\end{align*}
Hence,
\[
c^{(j)}_{k,\varepsilon}(t)=e^{-\lambda_k t} d_k^{(j)} \int_0^t e^{\lambda^{(j)}_k s} \chi^2_\varepsilon(s)ds, \quad t\ge 0,
\]
and $c^{(j)}_{k,\varepsilon}(t)$ is uniformly bounded in $k$, $\varepsilon$.
For $t\ge 1$ and $k=1,2,\dots$, fixed, we get
\begin{equation}
\label{eq_c_k}
c^{(j)}_{k,\varepsilon}(t)=e^{-\lambda^{(j)}_k t} d_k \int_0^1\chi^2(s)e^{\lambda^{(j)}_ks\varepsilon}ds\to e^{-\lambda^{(j)}_k t} d^{(j)}_k, \textrm{ as } \varepsilon \to 0. 
\end{equation}
Using \eqref{eq_c_k}, we may let $\varepsilon\to 0$ in \eqref{eq_bou_100}, and conclude that 
\begin{equation}
\label{eq_ser_1}
\sum_{k=1}^{\infty} e^{-\lambda_k^{(1)} t}d_k^{(1)} (\nu\cdot A_1\nabla \varphi_k^{(1)})|_{\p \Omega}=\sum_{k=1}^{\infty} e^{-\lambda_k^{(2)} t}d_k^{(2)} (\nu \cdot A_2\nabla \varphi_k^{(2)})|_{\p \Omega},\quad  t\ge 1. 
\end{equation}

In what follows we shall have to distinguish the eigenvalues of the operator $P_j$, $j=1,2$. In order to do that let us continue to denote by $\lambda_k^{(j)}$ the sequence of \emph{distinct} eigenvalues of $P_j$ and let $m_k^{(j)}$ denote the multiplicity of $\lambda_k^{(j)}$. 
Let 
$\varphi^{(j)}_{k,1}, \dots, \varphi^{(j)}_{k,m_{k}^{(j)}}$ be an orthonormal basis of the eigenfunctions, corresponding to the eigenvalue $\lambda_k^{(j)}$.
By the uniqueness of the Dirichlet series, see \cite{Choulli_book},  and  \eqref{eq_ser_1}, we obtain the following result.

\begin{prop}
\label{prop_almost_spec}
Assume that  $\Sigma_{\gamma,\kappa_1,A_1}=\Sigma_{\gamma,\kappa_2,A_2}$. Then  for all $k=1,2,\dots$, we have
\[
\lambda^{(1)}_k= \lambda^{(2)}_k,
\]
and
\begin{equation}
\label{eq_ser_2}
\sum_{i=1}^{m_k^{(1)}} d_{k,i}^{(1)} (\nu \cdot A_1\nabla \varphi_{k,i}^{(1)})|_{\p \Omega}=\sum_{i=1}^{m_k^{(2)}} d_{k,i}^{(2)} (\nu \cdot A_2\nabla \varphi_{k,i}^{(2)})|_{\p \Omega}.
\end{equation}
Here $d^{(j)}_{k,i}=(\gamma \nabla w^h\cdot \nabla w^{\tilde h},\varphi^{(j)}_{k,i})_{L^2}$ and $h,\tilde h\in H^{s}(\p \Omega)$ are arbitrary functions. 
\end{prop}

Let us introduce the following linear continuous  operators
\begin{align*}
R_k^{(j)}&:L^2(\Omega)\to L^2(\p \Omega),\\
R_k^{(j)}(F)&=\sum_{i=1}^{m_k^{(j)}} (F,\varphi_{k,i}^{(j)})_{L^2} (\nu\cdot A_j\nabla \varphi_{k,i}^{(j)})|_{\p \Omega}, \ j=1,2,\ k=1,2,\dots.
\end{align*}
Proposition \ref{prop_density_gradients} together with \eqref{eq_ser_2} implies that $R_k^{(1)}=R_k^{(2)}$ on a dense subset of $L^2(\Omega)$, and hence, everywhere.   
On the level of the distribution kernels, we obtain that for all $ k=1,2,\dots$,
\begin{equation}
\label{eq_sch_ker_2}
\sum_{i=1}^{m_k^{(1)}} \varphi_{k,i}^{(1)}(x) (\nu \cdot A_1\nabla \varphi_{k,i}^{(1)})(y)=\sum_{i=1}^{m_k^{(2)}} \varphi_{k,i}^{(2)}(x) (\nu \cdot A_2\nabla \varphi_{k,i}^{(2)})(y), x\in\Omega,y\in \p \Omega.
\end{equation}

We shall next  need the following result,  see \cite{Canuto_Kavian_2001}. Since the argument is short, for the convenience of the reader, we give it here. 

\begin{lem}
\label{lem_alg_1}
The functions $ (\nu\cdot A_j\nabla\varphi^{(j)}_{k,1})|_{\p \Omega}, \dots, (\nu\cdot A_j\nabla\varphi^{(j)}_{k,m_{k}^{(j)}})|_{\p \Omega}$ are linearly independent, $j=1,2$.
\end{lem}

\begin{proof}
Assume that there are $c_1,\dots,c_{m_k^{(j)}}\in \R$ such that
\[
\sum_{i=1}^{m_{k}^{(j)}} c_i  (\nu\cdot A_j\nabla\varphi^{(j)}_{k,i})|_{\p \Omega}=0. 
\]
Then $\varphi=\sum_{i=1}^{m_{k}^{(j)}} c_i \varphi^{(j)}_{k,i}$ satisfies
\begin{align*}
&P_j\varphi=\lambda_k^{(j)}\varphi\quad \text{in}\quad \Omega,\\
&\varphi|_{\p \Omega}=0,\quad
\nu \cdot A_j\nabla\varphi|_{\p \Omega}=0.
\end{align*}
By the unique continuation principle, we get $\varphi=0$ in $\Omega$. Thus,  $c_i=0$ for $1\le i \le m_{k}^{(j)}$. This proves the lemma. 

\end{proof}

Our next goal is to analyze the consequences of \eqref{eq_sch_ker_2}, and the key step here is the following algebraic result which is similar to \cite[Lemma 2.3]{Canuto_Kavian_2001}.

\begin{lem}
\label{lem_alg_2}
Let $f_i:\Omega\to \R$, $\tilde f_i:\p \Omega\to \R$, $i=1,\dots,m^{(1)}$, and $g_l:\Omega\to\R$, $\tilde g_l:\p \Omega\to \R$, $l=1,\dots,m^{(2)}$, be such that 
\begin{equation}
\label{eq_lem_alg}
\sum_{i=1}^{m^{(1)}} f_i(x)\tilde f_i(y)= \sum_{l=1}^{m^{(2)}} g_l(x)\tilde g_l(y)\quad \textrm{for all  }  x\in\Omega, y\in \p \Omega.
\end{equation}
Moreover,  assume that the systems  $\{f_1,\dots,f_{m^{(1)}}\}$, $\{\tilde f_1,\dots,\tilde f_{m^{(1)}}\}$, $\{g_1,\dots,g_{m^{(2)}}\}$ and  $\{\tilde g_1,\dots,\tilde g_{m^{(2)}}\}$  are all linearly independent. Then $m^{(1)}=m^{(2)}$ and 
there exists an $m^{(1)}\times m^{(1)}$ invertible matrix $T$ with real entries such that
\[
F(x)=TG(x),\quad \tilde F(y)=(T^t)^{-1}\tilde G(y)\quad \textrm{for all  } x\in \Omega, y\in \p \Omega.
\]
Here we use the notation
\begin{align*}
&F(x)=\begin{pmatrix} f_1(x)\\
\vdots\\
f_{m^{(1)}}(x)
\end{pmatrix}, \tilde F(y)=\begin{pmatrix} \tilde f_1(y)\\
\vdots\\
\tilde f_{m^{(1)}}(y)
\end{pmatrix},\\
 &G(x)=\begin{pmatrix} g_1(x)\\
\vdots\\
g_{m^{(2)}}(x)
\end{pmatrix}, \tilde G(y)=\begin{pmatrix} \tilde g_1(y)\\
\vdots\\
\tilde g_{m^{(2)}}(y)
\end{pmatrix}.
\end{align*}

\end{lem}

\begin{proof}
As $\tilde f_1$ is not identically zero in $\p \Omega$, there exists $y_1\in \p \Omega$ such that $\tilde f_1(y_1)\ne 0$.  
Assuming that  
\[
\det\begin{pmatrix} \tilde f_1(y_1) & \tilde f_2(y_1)\\
\tilde f_1(y) & \tilde f_2(y) 
\end{pmatrix} =0\quad  \textrm{for all  } y\in \p \Omega,
\]
we get that $\tilde f_1, \tilde f_2$ are linearly dependent which contradicts the assumptions of the proposition. Thus, there exists  $y_2\in \p\Omega$ such that
\[
\det\begin{pmatrix} \tilde f_1(y_1) & \tilde f_2(y_1)\\
\tilde f_1(y_2) & \tilde f_2(y_2) 
\end{pmatrix} \ne 0. 
\]
Continuing in the same way, we find points $y_1,y_2,\dots, y_{m^{(1)}}\in\p \Omega$ such that the matrix
\[
Q_{\tilde f}=\begin{pmatrix} \tilde f_1(y_1) & \tilde f_2(y_1)& \hdots & \tilde f_{m^{(1)}}(y_1) \\
\tilde f_1(y_2) & \tilde f_2(y_2) & \hdots & \tilde f_{m^{(1)}}(y_2)\\
\vdots & \vdots & \vdots & \vdots\\
\tilde f_1(y_{m^{(1)}}) & \tilde f_2(y_{m^{(1)}}) & \hdots & \tilde f_{m^{(1)}}(y_{m^{(1)}})
\end{pmatrix}
\]
is invertible. 
It follows from \eqref{eq_lem_alg} that $Q_{\tilde f}F(x)=Q_{\tilde g}G(x)$ for any $x\in \Omega$, where
\[
Q_{\tilde g}=\begin{pmatrix} \tilde g_1(y_1) & \hdots & \tilde g_{m^{(2)}}(y_1) \\
\vdots & \vdots & \vdots\\
\tilde g_1(y_{m^{(1)}}) &  \hdots & \tilde g_{m^{(2)}}(y_{m^{(1)}})
\end{pmatrix}.
\]
Thus, $F(x)=TG(x)$ for any $x\in \Omega$, where $T=Q_{\tilde f}^{-1}Q_{\tilde g}$, and therefore, $m^{(1)}\le m^{(2)}$. Similarly, using the fact that $\{\tilde g_1,\dots, \tilde g_{m^{(2)}}\}$ is linearly independent, we have $G(x)=T_1F(x)$ for any $x\in \Omega$, and  $m^{(2)}\le m^{(1)}$. Hence, $m^{(1)}= m^{(2)}$.  Furthermore, 
\begin{equation}
\label{eq_lem_alg_2}
F(x)=TT_1F(x) \quad \textrm{for all  } x\in \Omega.
\end{equation}
 Since the system $\{f_1,\dots, f_{m^{(1)}}\}$ is linearly independent, in the same way as above, we see that there are points $x_1,\dots, x_{m^{(1)}}\in \Omega$ such that the vectors $F(x_1),\dots, F(x_{m^{(1)}})$ form a basis in $\C^{m^{(1)}}$.  Thus, \eqref{eq_lem_alg_2} implies that $TT_1=I$ and therefore, $T$ is invertible. Similarly, one can see that $\tilde F(y)=\tilde T \tilde G(y)$ for all $y\in \p \Omega$ with an invertible matrix $\tilde T$. 

It follows from \eqref{eq_lem_alg} that $F(x)\cdot \tilde F(y)=G(x)\cdot \tilde G(y)$ and therefore, 
\[
(\tilde T^tT-I)G(x)\cdot \tilde G(y)=0 \quad \textrm{for all  } x\in \Omega, y\in \p \Omega.
\] 
Since there exist  points $x_1,\dots,x_{m^{(1)}}\in \Omega$ and $y_1,\dots,y_{m^{(1)}}\in \p \Omega$ such that  the vectors $G(x_1), \dots G(x_{m^{(1)}})$ (respectively, $\tilde G(y_1), \dots \tilde G(y_{m^{(1)}})$) form a basis in $\C^{m^{(1)}}$, we get $\tilde T^t T=I$.  This proves the claim. 
\end{proof}

If follows from Lemma \ref{lem_alg_2} together with 
\eqref{eq_sch_ker_2} and Lemma \ref{lem_alg_1} that $m^{(1)}_k=m^{(2)}_k=:m_k$, and there exists an $m_k\times m_k$ invertible matrix $T$ such that 
\begin{equation}
\label{eq_vect_alg_1}
\begin{pmatrix} \varphi^{(1)}_{k,1}(x)\\
\vdots\\
\varphi^{(1)}_{k,m_k}(x)
\end{pmatrix}=T \begin{pmatrix} \varphi^{(2)}_{k,1}(x)\\
\vdots\\
\varphi^{(2)}_{k,m_k}(x)
\end{pmatrix},\quad  x\in \Omega,
\end{equation}
and
\begin{equation}
\label{eq_vect_alg_2}
\begin{pmatrix} \nu \cdot A_1(y)\nabla \varphi^{(1)}_{k,1}(y)\\
\vdots\\
 \nu \cdot A_1(y)\nabla\varphi^{(1)}_{k,m_k}(y)
\end{pmatrix}= (T^t)^{-1} \begin{pmatrix}  \nu \cdot A_2(y)\nabla\varphi^{(2)}_{k,1}(y)\\
\vdots\\
 \nu \cdot A_2(y)\nabla\varphi^{(2)}_{k,m_k}(y)
\end{pmatrix},\quad  y\in \p \Omega.
\end{equation}
Using \eqref{eq_vect_alg_1} and \eqref{eq_vect_alg_2}, we have
\begin{equation}
\label{eq_orth_mat}
T^tT\begin{pmatrix} \nu \cdot A_1(y)\nabla \varphi^{(2)}_{k,1}(y)\\
\vdots\\
 \nu \cdot A_1(y)\nabla\varphi^{(2)}_{k,m_k}(y)
\end{pmatrix}= \begin{pmatrix}  \nu \cdot A_2(y)\nabla\varphi^{(2)}_{k,1}(y)\\
\vdots\\
\nu \cdot A_2(y)\nabla\varphi^{(2)}_{k,m_k}(y)
\end{pmatrix},\quad  y\in \p \Omega.
\end{equation}

The next step is to show that the matrix $T$ is in fact  orthogonal. This will follow once we establish the following result. 

\begin{prop}
If $\Sigma_{\gamma,\kappa_1,A_1}=\Sigma_{\gamma,\kappa_2,A_2}$ then $A_1|_{\p \Omega}=A_2|_{\p\Omega}$.
\end{prop}

\begin{proof}

 Consider the following elliptic boundary value problem,
\begin{equation}
\label{eq_ebvp}
\begin{aligned}
&P_ju=\kappa_j \gamma \nabla w^h\cdot \nabla w^{\tilde h}\quad\text{in}\quad \Omega,\\
&u|_{\p \Omega}=0,
\end{aligned}
\end{equation}
where $w^h$ (respectively, $w^{\tilde h}$) is the solution to the problem \eqref{eq_w_h} with the boundary source $h\in H^s(\p \Omega)$ (respectively, $\tilde h\in H^s(\p \Omega)$), $s\ge 1/2$ large enough.  We shall now return to the original notation, where each eigenvalue $\lambda^{(j)}_k$ of the operator $P_j$ is repeated according its multiplicity. 
Since $0\not\in\textrm{spec}(P_j)$,
 the problem \eqref{eq_ebvp} has the unique solution 
\[
u^{(j)}=\sum_{k=1}^\infty \frac{d_k^{(j)}}{\lambda_k^{(j)}}\varphi_k^{(j)},\quad  d_k^{(j)}=(\gamma \nabla w^h\cdot \nabla w^{\tilde h},\varphi_k^{(j)})_{L^2},
\]
with convergence in $H^2(\Omega)$. 
Thus,
\[
\nu\cdot A_{j} \nabla u^{(j)}|_{\p \Omega}=\sum_{k=1}^\infty \frac{d_k^{(j)}}{\lambda_k^{(j)}}\nu \cdot A_j \nabla\varphi_k^{(j)}|_{\p \Omega},
\]
and therefore, it follows from Proposition \ref{prop_almost_spec} that 
if 
 $\Sigma_{\gamma,\kappa_1,A_1}=\Sigma_{\gamma,\kappa_2,A_2}$, then 
\begin{equation} 
\label{eq_b_nA}
 \nu \cdot A_1 \nabla u^{(1)}|_{\p \Omega}=\nu \cdot A_2 \nabla u^{(2)}|_{\p \Omega}.
\end{equation} 
Define the continuous map
\[
\Psi^{(j)} : L^2(\Omega)\to L^2(\p \Omega),\quad 
F\mapsto \nu \cdot A_j \nabla u^{(F, j)}|_{\p \Omega},
\]
 $u^{(F,j)}$ is a solution to the problem
\begin{align*}
-\nabla\cdot (A_j\nabla u^{(F,j)})&=F\quad\text{in}\quad \Omega,\\
u^{(F,j)}|_{\p \Omega}&=0. 
\end{align*}
It follows from \eqref{eq_b_nA} together with Proposition \ref{prop_density_gradients} that $\Psi^{(1)}=\Psi^{(2)}$ on a dense subset of $L^2(\Omega)$, and thus, everywhere.  Hence, Proposition \ref{prop_A-boundary} in 
Appendix A implies that 
 $A_1|_{\p \Omega}=A_2|_{\p\Omega}$. 
 
 \end{proof}

 Now going back to equation \eqref{eq_orth_mat}, using Lemma \ref{lem_alg_1} we obtain that  $T$ is an orthogonal matrix.  
 Proposition \ref{prop_almost_spec} together with \eqref{eq_vect_alg_1}  
 gives 
  the following result. 

\begin{prop}
\label{prop_spectral}
Assume that $\Sigma_{\gamma,\kappa_1,A_1}=\Sigma_{\gamma,\kappa_2,A_2}$.  Let $\varphi^{(1)}_k$ be an orthonormal basis  in $L^2(\Omega,\kappa_1^{-1}dx)$ of the Dirichlet eigenfunctions of the operator $P_1$.  Then the Dirichlet eigenvalues $\lambda^{(1)}_k$ (respectively, $\lambda^{(2)}_k$) of the operator $P_1$ (respectively, $P_2$), counted with multiplicities,  satisfy $\lambda^{(1)}_k=\lambda^{(2)}_k$ for all $k$ and there exists an orthonormal basis  in $L^2(\Omega,\kappa_2^{-1}dx)$ of the Dirichlet eigenfunctions $\varphi^{(2)}_k$ of the operator $P_2$ such that $\varphi^{(1)}_k=\varphi^{(2)}_k$ for all $k$.
\end{prop}

We shall next show that  Proposition \ref{prop_spectral} yields that  $\kappa_1=\kappa_2$. 
Indeed, let us write
\[
\kappa_1=\sum_{k=1}^\infty c_k\varphi_k^{(1)}=\sum_{k=1}^\infty c_k\varphi_k^{(2)}, \quad \kappa_2=\sum_{k=1}^\infty d_k\varphi_k^{(2)},
\]
where the Fourier coefficients $c_k$ are given by
\[
c_k=\int_\Omega \kappa_1 \varphi_k^{(1)}\kappa_1^{-1}dx=\int_\Omega \varphi_k^{(2)}dx=\int_{\Omega}\kappa_2\varphi_k^{(2)}\kappa_2^{-1}dx=d_k.
\]
Thus, $\kappa_1=\kappa_2$.  It follows that $P_1u=P_2u$, for any $u\in C^\infty_0(\Omega)$, and we get $A_1=A_2$. 
The proof of Theorem \ref{thm_main} is complete.

Theorem \ref{thm_main_2} can be proven by exactly the same arguments presented in this section applied to  the problem 
\eqref{eq_heat} with the right hand sides of the form
\[
F(t,x)=\chi_\varepsilon^2(t)H(x), \quad H\in L^2(\Omega), 
\]
where $\chi_\varepsilon$ is given by \eqref{eq_chi_epsilon}.

\begin{appendix}
\section{Boundary reconstruction}

Let $\Omega\subset\R^n$, $n\ge 2$, be a bounded domain with $C^\infty$ boundary, and 
 $A(x)=(a_{jk}(x))$, $1\le j,k\le n$, be 
a real symmetric $n\times n$ matrix with $a_{j,k}(x)\in C^\infty(\overline{\Omega})$. Assume that there exists $C_0>0$ such that 
\[
\sum_{j,k=1}^n a_{jk}(x)\xi_j\xi_k\ge C_0|\xi|^2\quad \textrm{for all } (x,\xi)\in \overline{\Omega}\times \R^n.
\]
 Consider the following elliptic boundary value problem,
\begin{equation}
\label{eq_bvp_2}
\begin{aligned}
-\nabla\cdot (A\nabla u)&=F\quad\text{in}\quad \Omega,\\
u|_{\p \Omega}&=0. 
\end{aligned}
\end{equation}
For any $F\in L^2(\Omega)$, 
the problem \eqref{eq_bvp_2}
has a unique solution $u=u^F\in H^2(\Omega)\cap H^1_0(\Omega)$ and one can define the map,
\begin{equation}
\label{eq_PI}
\Psi:L^2(\Omega)\to L^2(\p \Omega),\quad 
\Psi (F)= \nu \cdot A \nabla u^F|_{\p \Omega},
\end{equation}
where $\nu$ is the unit outer normal to the boundary $\p \Omega$.  We note that the map $\Psi$ is sometimes used to model boundary measurements for optical tomography with diffusion approximation,
\cite{HSTFMT1994, HJWT1995, SAHD1995}.

We have the following proposition, which is closely related to the earlier boundary reconstruction results of a Riemannian metric from the Dirichlet--to--Neumann map  \cite{KohnVog_1984, LeeUhl89}.

\begin{prop} 
\label{prop_A-boundary}
 The knowledge of the map $\Psi$ given by \eqref{eq_PI} determines the values of $A$ on the boundary $\p \Omega$. 

\end{prop}

\begin{proof}

We shall recover the values of $A$ on the boundary by analyzing the distribution kernel of the map $\Psi$, obtained by  constructing a right parametrix for the boundary value problem \eqref{eq_bvp_2}. 
Let us denote $\mathcal{A}=-\nabla\cdot (A\nabla )$. 
Since $A(x)$ is a positive definite matrix, smoothly depending on $x$, we can view $\overline{\Omega}$ as a Riemannian manifold with boundary, equipped with the metric $G=A^{-1}$, $G=(g_{ij})$, $1\le i,j\le n$.  To construct a parametrix for  \eqref{eq_bvp_2}, we shall work locally near a boundary point.  Let $x_0\in \p \Omega$ and  introduce the boundary normal coordinates $y=(y',y_n)\in U$, $y'=(y_1,\dots,y_{n-1})$, centered at $x_0$. Here $U$ stands for some open neighborhood of $0$ in $\R^{n}$.  In terms of the boundary normal coordinates, locally near $x_0$, the boundary $\p \Omega$ is defined by $y_n=0$, and $y_n>0$ if and only if $x\in \Omega$.  In what follows, we shall write again $(x',x_n)$ for the boundary normal coordinates. 

In the boundary normal coordinates, the metric $G$ has the form
\[
G=\sum_{\alpha,\beta=1}^{n-1} g_{\alpha\beta}(x) dx_{\alpha}dx_{\beta} +(dx_n)^2,
\]
see \cite{LeeUhl89}, and the principal symbol of the operator $\mathcal{A}$ is given by
\[
a_0(x,\xi)=\xi_n^2+\sum_{\alpha,\beta=1}^{n-1}g^{\alpha\beta}(x)\xi_{\alpha}\xi_{\beta}. 
\]
Therefore,  the equation $a_0(x,\xi',\xi_n)=0$, $\xi'=(\xi_1,\dots,\xi_{n-1})$, has the solutions,
\begin{equation}
\label{eq_roots_lambda}
\xi_n=\lambda_\pm(x,\xi'),\quad \lambda_\pm(x,\xi')=\pm i 
\sqrt{\sum_{\alpha,\beta=1}^{n-1}g^{\alpha\beta}(x)\xi_{\alpha}\xi_{\beta}}. 
\end{equation}

We can view $\mathcal{A}$ as a linear continuous map in the space $\mathcal{D}'(U)$. 
In the boundary normal coordinates, the problem \eqref{eq_bvp_2} has the following form,
\begin{equation}
\label{eq_bvp_3}
\begin{aligned}
&\mathcal{A}u=F\quad\text{in}\quad \R^n_+=\{x\in\R^n:x_n>0\},\\
&u|_{x_n=0}=0. 
\end{aligned}
\end{equation}

Let 
\[
r_0(x,\xi)=a_0(x,\xi)^{-1}(1-\chi(\xi)), \quad x\in U,\quad \xi\in \R^n,
\]
where $\chi(\xi)\in C_0^\infty(\R^n)$, $\chi(\xi)=0$ for $|\xi|\ge 1$ and $\chi=1$ near $0$. The operator $\textrm{Op}(r_0)$ is a rough parametrix for the operator $\mathcal{A}$, which will be sufficient  for our purposes.  Here  we are using the classical quantization of a symbol $a\in S^k(U\times \R^n)$, which is given by 
\[
\textrm{Op}(a)u(x)=\frac{1}{(2\pi)^n}\int_{\R^n}\!\! \int_{\R^n} e^{i(x-y)\cdot \xi}a(x,\xi) u(y)dyd\xi.
\]
As usual, we say that $a\in S^k(U\times\R^n)$ if  locally uniformly in $x\in U$, we have
\[
|\p_x^\alpha \p_\xi^\beta a(x,\xi)|\le C_{\alpha\beta} \langle \xi\rangle^{k-|\beta|},\quad  \langle \xi\rangle=\sqrt{1+|\xi|^2}.
\]

Let $r_+$ be the operation of restriction from $\R^n$ to $\R_+^n$ and let $e_+$ be the operation of extension by zero from $\R_+^n$ to $\R^n$.

We shall construct  a right parametrix for the boundary value problem \eqref{eq_bvp_3} in the following form
\[
R(F)=r_+\textrm{Op}(r_0)(e_+ F)+ R_b(\psi).
\]
Here  
\[
\psi(x')=-\tau_0 r_+\textrm{Op}(r_0)(e_+ F), \quad \tau_0:u\mapsto u|_{x_n=0},
\]
and $R_b$ will be constructed as a right parametrix for the boundary value problem
\begin{equation}
\label{eq_bvp_4_new}
\begin{aligned}
&\mathcal{A}u=0\quad\text{in}\quad \R^n_+,\\
&u|_{x_n=0}=\psi(x'). 
\end{aligned}
\end{equation}
In what follows we shall suppress the operator $r_+$ from the notation, as this will cause no confusion.  

When constructing the operator $R_b$, we shall follow the standard approach in the theory of elliptic boundary value problems, see \cite{ChazPir_book}.
To this end,  let $\tilde \chi\in C^\infty_0(\R^{n-1})$ be such that  $\tilde \chi=1$ for $|\xi'|\le 1$. Notice that 
\[
(1-\tilde \chi(\xi'))\chi(\xi)=0.
\] 
Let $\sigma=\sigma(x,\xi')$ be a simple closed $C^1$ smooth curve in the upper half-plane $\Im \xi_n>0$,  which encircles the root $\lambda_+(x,\xi')$ in the positive sense. In what follows we may and will choose $\sigma$ so that it is independent of $x\in U$, depending on $\xi'$ only, i.e. $\sigma=\sigma(\xi')$.
When $\varphi\in C^\infty_0(U\cap\R^{n-1})$, we define the operator 
\begin{equation}
\label{eq_I_1_new}
(\Pi\varphi)(x)=\frac{1}{i(2\pi)^n }\int_{\R^{n-1}}e^{ix'\cdot\xi'} \hat \varphi(\xi') (1-\tilde \chi(\xi'))\bigg(\int_{\xi_n\in \sigma}\frac{e^{ix_n\xi_n}}{a_0(x,\xi)} d\xi_n\bigg) d \xi',\quad x_n\ge 0,
\end{equation}
where 
\[
\hat \varphi(\xi')=\int_{\R^{n-1}} e^{-iy'\cdot \xi'}\varphi(y')dy'
\]
is the Fourier transform of $\varphi$. 
By a contour deformation argument in the complex $\xi_n$--plane, we have
\[
\Pi\varphi=\frac{1}{i}\textrm{Op}((1-\tilde \chi)a_0^{-1})(\varphi\otimes \delta_{x_n=0}),\quad x_n>0.
\]
We get therefore,
\[
\mathcal{A}\Pi\varphi=\textrm{Op}(b)(\varphi\otimes \delta_{x_n=0}), \quad b\in S^{-1}(U\times\R^n), \quad x_n>0, 
\]
since the operator $\mathcal{A}$ is local.

We shall take $R_b(\psi)=\Pi \varphi$, for some function $\varphi$, defined locally near $0\in \R^{n-1}$, to be found  from the boundary condition, i.e.
\begin{equation}
\label{eq_bound_con_1}
\tau_0\Pi \varphi=\psi.
\end{equation}
To this end, 
we shall now prove that $\tau_0\Pi$ is an elliptic pseudodifferential operator on the boundary and compute its principal symbol. 
 By the residue calculus, using \eqref{eq_I_1_new}, we get
\begin{align*}
\tau_0 \Pi=\textrm{Op}(d_0),\quad d_0(x',\xi')=(1-\tilde \chi(\xi'))\frac{1}{2\pi i}\int_{\xi_n\in \sigma} \frac{1}{a_0(x',0,\xi)}d\xi_n\\
=(1-\tilde \chi(\xi'))\frac{1}{\p_{\xi_n}a_0(x',0,\xi',\lambda_+)}=(1-\tilde \chi(\xi'))\frac{1}{2i |\xi'|_A }\in S^{-1}((U\cap \R^{n-1})\times\R^{n-1}),
\end{align*}
where 
\[
|\xi'|_{A}=\sqrt{\sum_{\alpha,\beta=1}^{n-1}g^{\alpha\beta}(x',0)\xi_{\alpha}\xi_{\beta}}.
\]
We introduce next a rough parametrix of $\tau_0\Pi$, given by  $\textrm{Op}(\tilde d_0)$, where $\tilde d_0
\in S^1((U\cap \R^{n-1})\times\R^{n-1})$ is such that
\[
\tilde d_0=2i |\xi'|_A,\quad \textrm{for } |\xi'| \textrm{ large}.
\]
To satisfy  \eqref{eq_bound_con_1}, we choose  
\[
\varphi=\textrm{Op}(\tilde d_0)\psi=-\textrm{Op}(\tilde d_0)(\tau_0\textrm{Op}(r_0)(e_+ F)). 
\]
This choice of $\varphi$ completes the construction of  a rough parametrix for the boundary value problem \eqref{eq_bvp_4_new}, given by
\[
R_b(\psi)=-\Pi\textrm{Op}(\tilde d_0)(\tau_0\textrm{Op}(r_0)(e_+ F)).
\]
Hence, the parametrix for the problem \eqref{eq_bvp_3} has the form
\[
R(F)=\textrm{Op}(r_0)(e_+ F)-\Pi\textrm{Op}(\tilde d_0)(\tau_0\textrm{Op}(r_0)(e_+ F)). 
\]

In the boundary normal coordinates, the operator $\Psi$ is given by
\[
\Psi(F)=\tau_0 \p_{x_n} u,
\]
and therefore,  to obtain the claim of the proposition it suffices to analyze  the distribution kernel $K(x',y)$ of the operator 
$\tau_0 \p_{x_n} R$ given by
\[
(\tau_0 \p_{x_n} R(F))(x')=\int_{\R^n}K(x',y)F(y)dy,\quad x'\in \R^{n-1},\quad y\in \R^n.
\]
Let us first consider the Schwartz kernel of the operator
\begin{align*}
\tau_0 \p_{x_n} \textrm{Op}(r_0)(e_+ F)=\frac{1}{(2\pi)^n}\tau_0\p_{x_n} \int_{\R^n} e^{ix\cdot \xi}r_0(x,\xi)\hat{ e_+F}(\xi)d\xi,
\end{align*}
which 
is given by
\[
K^{(1)}(x',y)=\frac{1}{(2\pi)^n}\int_{\R^{n}} e^{i(x'-y')\cdot\xi'}e^{-iy_n\xi_n}(i\xi_nr_0(x',0,\xi)+\p_{x_n}r_0(x',0,\xi))d\xi.
\]
Recall that here $y_n> 0$. Restricting the attention to the region  $|\xi'|\ge 1$, by a contour deformation argument to the lower half plane, we find that 
\[
\int_{\R}e^{-iy_n\xi_n}i\xi_nr_0(x',0,\xi)d\xi_n=2\pi i
e^{-iy_n\lambda_-}\frac{i\lambda_-}{\p_{\xi_n}a_0(x',0,\xi',\lambda_-)}
=\pi e^{-y_n |\xi'|_A}, 
\]
and therefore,
\[
K^{(1)}(x',y)=\frac{1}{(2\pi)^{n-1}}\int_{\R^{n-1}} e^{i(x'-y')\cdot\xi'}  t(x',y_n,\xi') d\xi',
\]
where 
\[
t(x',y_n,\xi') =
\frac{e^{-y_n |\xi'|_A}}{2}+ e^{-y_n |\xi'|_A}\mathcal{O}\bigg(\frac{1}{|\xi'|_A}\bigg),\quad |\xi'| \ge 1. 
\]

Next,  the operator
$\tau_0\p_{x_n}\Pi$ is a pseudodifferential operator on $\R^{n-1}$, given by
\begin{align*}
&(\tau_0 \p_{x_n}\Pi) v(x')
=\frac{1}{(2\pi)^{n-1}}\\
&\int_{\xi'\in\R^{n-1}}(1-\tilde \chi(\xi')) e^{ix'\cdot \xi'}\hat v(\xi')\frac{1}{2\pi i}\int_{\xi_n\in \sigma}\bigg( \frac{i\xi_n }{a_0(x',0,\xi)} +\p_{x_n}\bigg(\frac{1}{a_0(x',0,\xi)}\bigg)\bigg)d\xi'd\xi_n.
\end{align*}
The principal symbol of the operator $\tau_0\p_{x_n}\Pi$ is therefore
\[
\frac{1}{2\pi i}\int_{\xi_n\in \sigma}  \frac{i\xi_n }{a_0(x',0,\xi)}d\xi_n=i\frac{\lambda_+}{\p_{\xi_n}a_0(x',0,\xi',\lambda_+)}= \frac{i}{2},\ |\xi'| \textrm{ large enough}.
\]
The operator 
$\textrm{Op}(\tilde d_0)$ is also a pseudodifferential operator on $\R^{n-1}$ and its principal symbol is given by
$2\lambda_+\in S^{1}((U\cap \R^{n-1})\times\R^{n-1})$,  $|\xi'|$ large enough. Hence, the principal symbol of the operator  $\tau_0\p_{x_n}\Pi \textrm{Op}(\tilde d_0)$ is $i\lambda_+\in S^{1}((U\cap \R^{n-1})\times\R^{n-1})$, and therefore, its kernel is given by
\begin{align*}
K_1^{(2)}(x',z')&=\frac{1}{(2\pi)^{n-1}}\int_{\R^{n-1}} e^{i(x'-z')\cdot \xi'}d_1(x',\xi')d\xi',\\
d_1(x',\xi')&=i\lambda_+(x',0,\xi')+p_{0}(x',\xi'), \quad |\xi'| \textrm{ large enough},
\end{align*}
where $p_{0}\in S^{0}((U\cap \R^{n-1})\times \R^{n-1})$.  

Finally, the kernel of the operator $\tau_0\textrm{Op}(r_0)$ is given by
\begin{align*}
K_2^{(2)}(z',y)&=\frac{1}{(2\pi)^n}\int_{\R^n} e^{i(z'-y')\cdot\eta'}e^{-iy_n\eta_n}r_0(z',0,\eta)d\eta\\
&=\frac{1}{(2\pi)^{n-1}}\int_{\R^{n-1}} e^{i(z'-y')\cdot\eta'} d_2(z',y_n,\eta')d\eta',\\
d_2(z',y_n,\eta')&=\frac{1}{2}\frac{e^{-y_n |\eta'|_A}}{|\eta'|_A},\quad |\eta'| \textrm{ large enough}.
\end{align*}
Here as usual we use the residue calculus, where only the pole in the lower half plane contributes. 

Hence, the kernel of the composition $\tau_0\p_{x_n}\Pi \textrm{Op}(\tilde d_0)\tau_0\textrm{Op}(r_0)(e_+ F)$ is given by
\begin{align*}
K^{(2)}(x',y)&=\int K_1^{(2)}(x',z')K_2^{(2)}(z',y)dz'\\
&=\frac{1}{(2\pi)^{2(n-1)}}\int\!\!\!\int\!\!\!\int e^{i(x'-z')\cdot \xi'} e^{i(z'-y')\cdot \eta'}  d_1(x',\xi')d_2(z',y_n,\eta')d\xi' d\eta' dz'\\
&=
\frac{1}{(2\pi)^{n-1}}\int e^{i(x'-y')\cdot \eta'} c(x',y_n,\eta')d\eta',
\end{align*}
where 
\begin{align*}
c(x',y_n,\eta')=\frac{1}{(2\pi)^{n-1}}\int\!\!\!\int e^{i(x'-z')\cdot(\xi'-\eta')}d_1(x',\xi')d_2(z',y_n,\eta')d\xi' dz'.
\end{align*}
Here $y_n\ge 0$ occurs as a parameter. 
Now $d_1(x',\xi')\in S^{1}((U\cap \R^{n-1})\times\R^{n-1})$ for large $|\xi'|$ and $d_2(z',y_n,\eta')$
satisfies
\[
|\p_{z'}^{\alpha}\p_{\eta'}^\beta \p_{y_n}^{\gamma} d_2(z',y_n,\eta')|\le C_{\alpha,\beta,\gamma}e^{-y_n\langle \eta'\rangle} \langle \eta'\rangle^{-1-|\beta|+|\gamma|},
\]
where $\langle \eta'\rangle=\sqrt{1+|\eta'|_A^2}$ is large enough. Hence, $c(x',y_n,\eta')$, depending on the parameter $y_n$,  is the symbol of the composition of two pseudodifferential operators in the tangential directions. By the standard results on pseudodifferential operators, see 
\cite{Grigis_Sjostrand_book}, it has the following asymptotic expansion,
\[
c(x',y_n,\eta')\sim \sum_{|\alpha|\ge 0} \frac{1}{\alpha!} \p_{\eta'}^\alpha d_1(x',\eta') D_{x'}^\alpha d_2(x',y_n,\eta'),
\]
with the leading term $d_1(x',\eta')d_2(x',y_n,\eta')$. 
The knowledge of the operator $\Psi$ implies  the knowledge of the kernel $K^{(1)}(x',y)-K^{(2)}(x',y)$ for any $x'\in U\cap \R^{n-1}$ and $y\in U\cap \overline{\R_+^n}$. This implies the knowledge of 
\[
\frac{e^{-y_n |\xi'|_A}}{2}+ e^{-y_n |\xi'|_A} \mathcal{O}\bigg(\frac{1}{|\xi'|_A}\bigg)-c(x',y_n,\xi'),\quad \textrm{for any } |\xi'|\textrm{ large enough}.
\]
The leading term of the latter expression is given by 
\[
\frac{e^{-y_n |\xi'|_A}}{2}+\frac{e^{-y_n |\xi'|_A}}{2}=e^{-y_n |\xi'|_A}.
\]
Varying $\xi'$,  we recover $A(x',0)$. The proof is complete.  

\end{proof}

\end{appendix}

\section*{Acknowledgements}  
The research of K.K. was financially supported by the
Academy of Finland (project 125599).  The research of M.L. and S.S. was financially supported by the Academy of Finland (Center of Excellence programme 213476 and Computational Science Research Programme, project 134868). 
This project  was partially conducted at  the
Mathematical Sciences Research Institute,  Berkeley, whose hospitality is gratefully acknowledged.


\begin{thebibliography} {1}

\bibitem{ABCTF}
Ammari, H.,  Bonnetier, E., Capdeboscq, Y., Tanter, M., and  Fink, M.,  \emph{Electrical impedance tomography by elastic deformation},  SIAM J. Appl. Math.  \textbf{68}  (2008),  no. 6, 1557--1573.

\bibitem{Ammari2}
Ammari, H.,   Capdeboscq, Y..,   Kang, H.,  and  Kozhemyak, A., 
\emph{Mathematical models and reconstruction methods in magneto-acoustic imaging}
European Journal of Applied Mathematics, \textbf{20} (2009),  303--317.

\bibitem{AstPai_2006}
Astala, K.,  P\"aiv\"arinta, L.,  \emph{Calder\'on's inverse conductivity problem in the plane},  Ann. of Math. (2)  \textbf{163}  (2006),  no. 1, 265--299. 

\bibitem{AstPaiLassas_2005}
Astala, K.,  P\"aiv\"arinta, L.,   and Lassas, M.,  \emph{Calder\'on's  inverse problem for anisotropic conductivity in the plane},  Comm. Partial Differential Equations  \textbf{30}  (2005),  no. 1--3, 207--224. 

\bibitem{BalUhl_2010}
Bal, G.,  Uhlmann, G., \emph{Inverse Diffusion Theory of Photoacoustics},
Inverse Problems, \emph{26} (2010), 085010. 

\bibitem{BukhKlib_1981}
Bukhgeim,  A.,   Klibanov, M., \emph{Uniqueness in the large of a class of multidimensional inverse problems}, (Russian)
Dokl. Akad. Nauk SSSR,  \textbf{260} (1981), no. 2, 269--272. 

\bibitem{Canuto_Kavian_2001} Canuto, B., Kavian, O.,  \emph{Determining coefficients in a class of heat equations via boundary measurements},  SIAM J. Math. Anal.  \textbf{32}  (2001),  no. 5, 963--986.

\bibitem{ChazPir_book}
Chazarain, J.,  Piriou, A., \emph{Introduction to the theory of linear partial differential equations},
Translated from the French. Studies in Mathematics and its Applications, 14. North-Holland Publishing Co., Amsterdam-New York, 1982.



\bibitem{Choulli_book}
Choulli, M., \emph{Une introduction aux probl\`emes inverses elliptiques et
  paraboliques}, volume~65 of {\em Math\'ematiques \& Applications (Berlin)
  [Mathematics \& Applications]},  Springer-Verlag, Berlin, 2009.


\bibitem{GreKurLasUhl2009}
Greenleaf, A., Kurylev, Y., Lassas, M., and  Uhlmann, G., \emph{Invisibility and
inverse problems},  Bull. Amer. Math. Soc. (N.S.)  \textbf{46}  (2009),  no. 1, 55--97.



\bibitem{GreKurLasUhl2007}
Greenleaf, A., Kurylev, Y., Lassas, M., and  Uhlmann, G., \emph{Full-wave invisibility of active devices at all frequencies},  Comm. Math. Phys.
\textbf{275} (2007), no. 3, 749--789. 

\bibitem{GreLasUhl2003MRL}
Greenleaf, A.,  Lassas, M., and  Uhlmann, G.,  \emph{On nonuniqueness for Calder\'on's inverse problem},  Math. Res. Lett.  \textbf{10}  (2003),  no. 5--6, 685--693. 

\bibitem{GreLasUhl2003}
Greenleaf, A.,  Lassas, M.,  and Uhlmann, G.,  \emph{The Calder\'on problem for conormal potentials. I. Global uniqueness and reconstruction},  Comm. Pure Appl. Math.  \textbf{56}  (2003),  no. 3, 328--352.

\bibitem{Grigis_Sjostrand_book}
Grigis, A., Sj\"ostrand, J.,  \emph{Microlocal analysis for differential operators. An introduction}, London Mathematical Society Lecture Note Series, 196. Cambridge University Press, Cambridge, 1994.

\bibitem{GebSch2008} Gebauer, B., Scherzer, O., \emph{Impedance-acoustic tomography},
  SIAM J. Appl. Math., \textbf{69} (2008),  no. 2, 565--576. 


\bibitem{Grubbbook2009} Grubb, G., \emph{Distributions and operators},
    Graduate Texts in Mathematics, volume  252, Springer, New York, 2009. 
   

\bibitem{HSTFMT1994}
Haskell, R., Svaasand, L., Tsay, T., Feng,T.,   McAdams, M., and  Tromberg, B.,  \emph{Boundary conditions for the diusion equation in radiative transfer},  J.
Opt. Soc. Am. A., \textbf{11} (1994), 2727--2741. 
  
\bibitem{HJWT1995}    
Hielscher, A.,  Jacques, S., Wang, L.,  and Tittel, F.,  \emph{The influence of boundary conditions on the accuracy of diffusion theory in time-resolved reflectance spectroscopy of biological tissue}, Phys. Med. Biol. \textbf{40}(1995), 1957--1975. 
 
\bibitem{ImanYam_1998} Imanuvilov, O.,  Yamamoto, M., \emph{Lipschitz stability in inverse parabolic problems by the Carleman estimate}, 
Inverse Problems \textbf{14} (1998), no. 5, 1229--1245.  

\bibitem{Isak_1999} Isakov, V., \emph{Some inverse problems for the diffusion equation}, Inverse Problems \textbf{15} (1999), no. 1, 3--10. 
  
  
\bibitem{KKL_book} Katchalov, A.,  Kurylev, Y., and  Lassas, M., \emph{Inverse boundary spectral problems}, 
Chapman \& Hall/CRC Monographs and Surveys in Pure and Applied Mathematics, 123, 2001.

\bibitem{KKLM_2004}
Katchalov, A.,  Kurylev, Y.,  Lassas, M.,  and Mandache, N., \emph{Equivalence of time-domain inverse problems and boundary spectral problems}, Inverse Problems \textbf{20} (2004), no. 2, 419--436. 
    
\bibitem{Klib_1992}
Klibanov, M., \emph{Inverse problems and Carleman estimates}, Inverse Problems  \textbf{8}  (1992),  no. 4, 575--596.
    
\bibitem{KohnVog_1984} Kohn, R.,  Vogelius, M., \emph{Determining conductivity by boundary measurements}, Comm. Pure Appl. Math. \textbf{37} (1984), no. 3, 289--298.

\bibitem{KuKu}
 Kuchment, P., Kunyansky, L.,  \emph{Mathematics of thermoacoustic tomography},   European J. Appl. Math.  \textbf{19}  (2008),  no. 2, 191--224. 
 
\bibitem{MREIT}
Kwon, O., Woo,  E.,  Yoon, J.,  and Seo, J., \emph{Magnetic resonance electrical impedance tomography (MREIT): simulation study of J-substitution
algorithm}, IEEE Trans. Biomed. Eng., \textbf{49} (2002), no. 2, 160--167.


\bibitem{LasUhl01}
Lassas, M.,  Uhlmann, G.,  \emph{On determining a Riemannian manifold from the
Dirichlet-to-Neumann map}.  Ann. Sci. \'Ecole Norm. Sup. (4)  \textbf{34}  (2001),  no. 5, 771--787.

\bibitem{LasTayUhl03}
Lassas, M., Taylor, M., and Uhlmann, G., \emph{The Dirichlet-to-Neumann map for
complete Riemannian manifolds with boundary},  Comm. Anal. Geom.  \textbf{11}  (2003),  no. 2,
207--221.



\bibitem{LeeUhl89} Lee, J.,  Uhlmann, G.,  \emph{Determining anisotropic real-analytic conductivities by
              boundary measurements}, Comm. Pure Appl. Math.,  \textbf{42} (1989), no. 8, 1097--1112. 

\bibitem{MaHe}
Ma, Q., and  He, B.,  \emph{Investigation on magnetoacoustic signal generation with magnetic induction
and its application to electrical conductivity reconstruction}, Phys. Med. Biol., \textbf{52} (2007), 5085--5099.

\bibitem{Nach_1988} Nachman, A., \emph{Reconstructions from boundary measurements},  Ann. of Math. (2)  \textbf{128}  (1988),  no. 3, 531--576.

\bibitem{Nach_1996} Nachman, A.,  \emph{Global uniqueness for a two-dimensional inverse boundary value problem},   Ann. of Math. (2)  \textbf{143}  (1996),  no. 1, 71--96.

\bibitem{NTT}
Nachman, A.,  Tamasan, A.,  and Timonov, A., \emph{Conductivity imaging with a single measurement of boundary and interior data},
Inverse Problems, \textbf{23} (2007), 2551--2563. 


\bibitem{PaiPanUhl}
P\"aiv\"arinta, L., Panchenko, A., and Uhlmann, G., \emph{Complex geometrical optics solutions for Lipschitz conductivities},  Rev. Mat. Iberoamericana \textbf{19} (2003), no. 1, 57--72. 


\bibitem{SAHD1995}
Schweiger, M.,   Arridge, S.,  Hiraoka, M,  and  Delpy, D.,  \emph{The finite element method for the propagation of light in scattering media: Boundary and source conditions},  Med Phys,  \textbf{22} (1995), 1779--1792.

\bibitem{Syl_1990}
Sylvester, J.,  \emph{An anisotropic inverse boundary value problem},  Comm. Pure Appl. Math.  \textbf{43}  (1990),  no. 2, 201--232.

\bibitem{SylUhl1986}
Sylvester, J.,  Uhlmann, G., \emph{A uniqueness theorem for an inverse boundary value problem in electrical prospection}, Comm. Pure Appl. Math. \textbf{39} (1986), no. 1, 91--112.

\bibitem{SylUhl1987}
    Sylvester, J.,  Uhlmann, G., \emph{A global uniqueness theorem for an inverse boundary value
              problem},
   Ann. of Math. (2), \textbf{125} (1987), no. 1, 153--169. 
   
\bibitem{Uhl2009}  Uhlmann, G., \emph{Electrical impedance tomography and Calder\'on's problem}, 
Inverse Problems, \textbf{25} (2009), 123011. 


\bibitem{Yamom2009} Yamamoto, M., \emph{Carleman estimates for parabolic equations and applications}, 
Inverse Problems, \textbf{25} (2009), 123013.  

\end{thebibliography}
\end{document}